\documentclass[10pt,draft,reqno]{amsart}
     \makeatletter
     \def\section{\@startsection{section}{1}%
     \z@{.7\linespacing\@plus\linespacing}{.5\linespacing}%
     {\bfseries
     \centering
     }}
     \def\@secnumfont{\bfseries}
     \makeatother
\newtheorem{theorem}{Theorem}[section]

\newtheorem{proposition}[theorem]{Proposition}

\theoremstyle{definition}
\newtheorem{definition}[theorem]{Definition}
\newtheorem{example}[theorem]{Example}
\theoremstyle{remark}
\newtheorem{remark}[theorem]{Remark}
\numberwithin{equation}{section} 

\newcommand{\ombar}{\overline{\Omega}}
 \newcommand{\po}{\partial\Omega}
 \newcommand{\me}{\mathcal E}
 \newcommand{\mf}{\mathcal F}
 \newcommand{\mfm}{\mathcal F^{\mu}}
 \newcommand{\mem}{\mathcal E_{\mu}}
 \newcommand{\Huntild}{\widetilde{H}^1(\Omega)}
\newcommand{\pmu}{\mathcal P_t^{\mu}}
\newcommand{\ma}{\mathcal A}
\newcommand{\emu}{e^{-t\Delta_{\mu}}}

 \newcommand{\Cap}{\mathrm{Cap}_{\ombar}}
\newcommand{\dmu}{\Delta_{\mu}}
\newcommand{\emup}{e^{-t\Delta_{\mu^+}}}
\newcommand{\emun}{e^{-t\Delta_{-\mu^-}}}
 \newcommand{\memp}{\mathcal E_{\mu^+}}
 \newcommand{\memn}{\mathcal E_{-\mu^-}}
\newcommand{\mfmp}{\mathcal F^{\mu^+}}
\begin{document}

\title[]
 {The Laplacian with Robin Boundary Conditions involving signed measures
}

\author{ AKHLIL KHALID *}

\address{Institute of Applied Analysis, University of Ulm, 89069 Ulm, Germany.}

\email{khalid.akhlil@uni-ulm.de}

\thanks{* The auhtor is supported by ``Deutscher Akademischer Austausch Dienst``(German Academic Exchange Service)}

\thanks{}

\subjclass[2000]{31C15, 31C25, 47D07, 60H30, 60J35, 60J60, 60J45}

\keywords{Dirichlet forms;Kato class of measures;Robin boundary conditions}

\date{}

\dedicatory{}

\commby{}


\begin{abstract}

In this work we propose to study the general Robin boundary value problem involving signed smooth measures on an arbitrary domain $\Omega$ of $\mathbb R^d$. A Kato class of measures is defined to insure the closability of the associated form $(\mem,\mfm)$. Moreover, the associated operator $\Delta_{\mu}$ is a realization of the Laplacian on $L^2(\Omega)$. In particular, when $|\mu|$ is locally infinite everywhere on $\po$, $\Delta_{\mu}$ is the laplacian with Dirichlet boundary conditions. On the other hand, we will prove that he semigroup $(\emu)_{t\geq 0}$ is sandwitched between $(\emup)_{t\geq 0}$ and $(\emun)_{t\geq 0}$ and we will see that the converse is also true.

\end{abstract}

\maketitle

\section{Introduction}

This Paper is a complement of our first paper \cite{A}, where we have studied Robin Laplacian in arbitrary domains involving positive smooth measures. Here we want to carry out the case of signed smooth measures.

As for Kato class defined in the context of perturbation of Dirichlet forms, one should define a specific Kato class adapted to the treatment of the Robin Laplacian. One should have in mind that this case is in fact a perturbation of the Neumann boundary condition by a certain measure.

In the Litterature, the first who had defined such class of measures was V. G. Papanicolau \cite{P}. His aim was to give the probabilistic solution of the schr\"odinger operator $-\Delta +V$ with Robin boundary condition $\frac{\partial}{\partial\nu}+\beta$ on the boundary $\po$, where $\Omega$ is bounded domain with $C^3$ boundary, and $\nu$ the outward unit normal vector on $\po$. The Borel function $\beta$ belong to a specific Kato class $\Sigma(\po)$, which means that,
\begin{equation}\label{eq:kat}
 \lim\limits_{t\downarrow 0}\sup_{x\in\ombar}\mathrm E^x\left[\int_0^t|\beta(X_s)|d L_s\right]=0
\end{equation}
where $L$ is the boundary local time of the Standard reflecting Brownian motion $X$ on $\ombar$.

With the same smootness assymption on the domain as above, Ma and Song \cite{MS,MS2} worked with a generalized Kato class of measure on $\ombar$ to study in a probabilistic point of view, the third boundary value problems, semilinear and generalized mixed boundary value problems. Ramasubramanian in \cite{R} remarks that, one can generalize the treatment in \cite{P} to bounded Lipschitz domains.

There is no study of the Robin boundary value problems in an arbitrary domain involving signed smooth measures on the boundary. There is two reasons for this: First, one need the reflecting Brownian motion $X$ on $\ombar$, which is defined to be the Hunt process associated with the Dirichlet form
\begin{equation}
 \me(u,v)=\int_{\Omega}\nabla u\nabla v dx,\quad\forall u,v\in H^1(\Omega)
\end{equation}

The dirichlet form $(\me,H^1(\Omega))$ need not to be regular, and then nothing insure the existence of $X$. Moreover, one can not define, for a ``bad'' $\Omega$, the capacity induced by $(\me,H^1(\Omega))$, and then to be able to reproduce the theory of perturbation of regular Dirichlet forms \cite{AM1,AM2,BM,FOT,S,SV} in our special case.

Throught \cite{FOT} for example,the form $(\me,\mf)$ is a regular Dirichlet form on 
\\$L^2(X,m)$ , where $X$ is a locally compact separable metric space, and $m$ a positive Radon measure on $X$ with $\mathrm{supp}[m]=X$.

For our purposes we take as in \cite{AW2} $X=\ombar$, where $\Omega$ is an open subset of $\mathbb R^d$, and the measure $m$ on the $\sigma-$algebra $\mathcal B(X)$ is given by $m(A)=\lambda(A\cap\Omega)$ for all $A\in\mathcal B(X)$ with $\lambda$ the Lebesgue measure, it follows that $L^2(\Omega)=L^2(X,\mathcal B(X),m)$, and we define a regular Dirichlet form $(\me,\mf)$ on $L^2(\Omega)$ by:

\[
 \me(u,v)=\int_{\Omega}\nabla u\nabla v dx\quad,
\forall u,v\in\mf
\]
where $\mf=\Huntild$ is the closure of $H^1(\Omega)\cap C_c(\overline{\Omega})$ in $H^1(\Omega)$. The domain $\widetilde{H}^1(\Omega)$ is so defined to insure the Dirichlet form $(\me,\mf)$ to be regular. In the special case where $\Omega$ is bounded with Lipschitz boundary, we have $\widetilde{H}^1(\Omega)=H^1(\Omega)$.

In \cite{A}, we have considered a perturbation on the boundary by a positive smooth measure. Here, we consider a purturbation on the boundary by signed smooth measure, we define then, for $\mu\in S(\po)-S(\po)$ 
\[
 \me_{\mu}(u,v)=\int_{\Omega}\nabla u\nabla v dx+\int_{\po}uvd\mu,\text{  }\forall u,v\in \mf^{\mu}
\]
where $\mfm=\Huntild\cap L^2(\po,|\mu|)$.

More precisely, we define a particular Kato class of measures, adapted to our context, we give also some of its properties and its analytic description, this is the subject of section 3. In section 4, we consider the Robin problem involving signed smooth measures. We will see that when $\mu\in S(\po)-S_K(\po)$, the Dirichlet form $(\mem,\mfm)$ is closed and the associated selfadjoint operator $\Delta_{\mu}$ is a realization of the Laplacian on $L^2(\po)$. In the special case where $|\mu|$ is locally infinite on $\po$, then $\Delta_{\mu}$ is the Laplacian with Dirichlet boundary conditions. Moreover,  $(\mem,\mfm)$ is regular if and only if $|\mu|$ is a Radon measure. In section 4, we will prove a domination theorem. It says that the semigroup $(\emu)_{t\geq 0}$ is sandwitched between $(\emup)_{t\geq 0}$ and $(\emun)_{t\geq 0}$. We will see that the converse is also true. That means that if one have a semigroup $(T(t))_{t\geq 0}$ sandwitched between $(\emup)_{t\geq 0}$ and $(\emun)_{t\geq 0}$, then $T(t)=e^{-t\Delta_{\nu-\mu^-}}$, where $\nu$ is a Radon measure charging no set of zero relative capacity.

\section{Preliminaries}

This section is devoted to preparations for the next sections. More precisily, it concerns the notion of relative capacity, smooth measures concentrated on the boundary $\po$, and the revuz correspondence between this smooth measures and positive additive functionals supported also by $\po$: thanks to the fact that the support of an additive functional is the relative quasi-support of its associated measure. In the two last subsections, we will define a general reflected Brwnian motion adapted to our context, and finally we review the case of positive smooth measures studied in \cite{A}.

We start with the Regular Dirichlet form $(\me,\mf)$ on $L^2(\Omega)$ defined by
\begin{equation}\label{eq:fe}
 \me(u,v)=\int_{\Omega}\nabla u\nabla v dx\quad,
\forall u,v\in\mf
\end{equation}
 where $\mf=\Huntild$, and we denote for any $\alpha>0:\text{
}\me_{\alpha}(u,v)=\me(u,v)+\alpha(u,v),\\\text{ }\forall
u,v\in\mf$.

\subsection{ Relative Capacity} The relative capacity is introduced
in a first time in \cite{AW1} to study the Laplacian with general
Robin boundary conditions on arbitrary domains. It is a special
case of the capacity associated with a regular Dirichlet form as
described in chapter 2 of \cite{FOT}. It seems to be an efficient
tool to analyse the phenomena occurring on the boundary $\po$ of
$\Omega$.

The relative capacity which we denote by
$\Cap$ is defined on a subsets of
$\overline{\Omega}$ by: For $A\subset\overline{\Omega}$ relatively
open (i.e. open with respect to the topology of
$\overline{\Omega}$) we set:
\[
 \Cap(A):=\text{inf}\{\me_1(u,u):u\in\widetilde{H}^1(\Omega):u\geq 1\text{ a.e on
}A\}
\]

And for arbitrary $A\subset\overline{\Omega}$, we set:
\[
 \Cap(A):=\text{inf}\{\Cap(B):B
\text{ relatively open }A\subset B\subset\overline{\Omega}\}
\]

A set $N\subset\overline{\Omega}$ is called a relatively polar if
$\Cap(N)=0$.

The relative capacity has the properties of a
capacity as described in \cite{FOT}. In particular, $\Cap$
is also an outer measure (but not a Borel measure) and a Choquet
Capacity.

A statement depending on $x\in A\subset\overline{\Omega}$ is said
to hold relatively quasi-everywhere (r.q.e.) on $A$, if there
exist a relatively polar set $N\subset A$ such that the statement
is true for every $x\in A\setminus N$.

Now we may consider functions in $\widetilde{H}^1(\Omega)$ as
defined on $\overline{\Omega}$, and we call a function
$u:\overline{\Omega}\rightarrow\mathbb R$ relatively
quasi-continuous (r.q.c.) if for every $\epsilon>0$ there exists a
relatively open set $G\subset\overline{\Omega}$ such that
$\Cap(G)<\epsilon$ and
$u|_{\overline{\Omega}\setminus G}$ is continuous.

It follows \cite{Wa}, that for each $u\in\Huntild$ there exists a
relatively quasi-continuous function
$\widetilde{u}:\ombar\rightarrow\mathbb R$ such that
$\widetilde{u}(x)=u(x)$ $m-$a.e. This function is unique
relatively quasi-everywhere. We call $\widetilde{u}$ the
relatively quasi-continuous representative of $u$.

For more details, we refer the reader to \cite{AW1,Wa}, where the
relative capacity is investigated, as well as its relation to the
classical one
\subsection{Revuz corespondence}
All families of measures on $\po$ defined in this subsection, was
originally defined on $X$ \cite{FOT}, and then in our settings on
$X=\ombar$, as a special case. We put $\po$ between brackets to recall our context, and
we keep in mind that the same things are valid if we put $\Omega$
or $\ombar$ instead of $\po$.

Let $\Omega\subset\mathbb R^N$ be open. A positive Radon measure
$\mu$ on $\partial\Omega$ is said to be of finite energy integral
if $$\int_{\partial\Omega}|v(x)|\mu(dx)\leq C\sqrt{\mathcal
E_1(v,v)}\quad,v\in\mathcal F\cap C_c(\overline{\Omega})$$for some
positive constant $C$. A positive Radon measure on
$\partial\Omega$ is of finite energy integral if and only if there
exists, for each $\alpha>0$, a unique function
$U_{\alpha}\mu\in\mathcal F$ such that $$\mathcal
E_{\alpha}(U_{\alpha}\mu,v)=\int_{\partial\Omega}v(x)\mu(dx)$$ We
call $U_{\alpha}\mu$ an $\alpha-$potential.

We denote by $S_0(\partial\Omega)$, the family of all positive
Radon measures of finite energy integral. We recall that each measure in $S_0(\partial\Omega)$ charges no set of zero
relative capacity.

We now turn to a class of measures larger than
$S_0(\po)$. Let us call a (positive) Borel measure $\mu$ on $\po$
smooth if it satisfies the following conditions:

- $\mu$ charges no set of zero relative capacity.

- There exist an increasing sequence $(F_n)_{n\geq 0}$ of closed
sets of $\po$ such that: \begin{equation}
\mu(F_n)<\infty\quad,n=1,2,...\end{equation}

\begin{equation} \lim_{n\rightarrow +\infty}Cap_{\ombar}(K\setminus F_n)=0\text{
for any compact } K\subset\po
\end{equation}

We denote by $S(\po)$ the family of all smooth measures on $\po$. The class
$S(\po)$ is quiet large and it contains all positive Radon measure
on $\po$ charging no set of zero relative capacity. There exist
also, by Theorem 1.1 \cite{AM2} a smooth measure $\mu$ on $\po$ (
hence singular with respect to $m$) "nowhere Radon" in the sense
that $\mu(G)=\infty$ for all non-empty relatively open subset $G$
of $\po$ (See Example 1.6\cite{AM2}).

Now we turn our attention to the correspondence between smooth
measures and additive functionals, known as Revuz correspondence.
As the support of an additive functional is the quasi-support of
its Revuz measure, we restrict our attention, as for smooth
measures, to additive functionals supported by $\po$. Recall that
as the Dirichlet form $(\me,\mf)$ is regular, then there exists a
Hunt process $M=(\Xi,X_t,\xi,P_x)$ on $\ombar$ which is
$m-$symmetric and associated with it.

\begin{definition}A function
$A:[0,+\infty[\times\Xi\rightarrow[-\infty,+\infty]$ is said to be
an Additive functional (AF) if:

1) $A_t$ is $\mathcal F_t-$measurable.

2) There exist a defining set $\Lambda\in\mathcal F_{\infty}$ and
an exceptional set $N\subset \po$ with $\Cap(N)=0$ such
that $P_x(\Lambda)=1$, $\forall x\in \Omega\setminus N$,
$\theta_t\Lambda\subset\Lambda$, $\forall t>0$;
$\forall\omega\in\Lambda, A_0(\omega)=0$; $|A_t(\omega)|<\infty$
for $t<\xi$. $A_{.}(\omega)$ is right continuous and has left
limit, and $ A_{t+s}(\omega)=A_t(\omega)+A_s(\theta_t\omega)\text{
}s,t\geq 0$\end{definition}

An additive functional is called positive continuous (PCAF) if, in
addition, $A_t(\omega)$ is nonnegative and continuous for each
$\omega\in\Lambda$. The set of all PCAF's on $\po$ is denoted
$\ma_c^{+}(\po)$.

Two additive functionals $A^1$ and $A^2$ are said to be equivalent
if for each $t>0$, $P_x(A^1_t=A^2_t)=1\text{ r.q.e }x\in \ombar$.

We say that $A\in\ma_c^+(\po)$ and $\mu\in S(\po)$ are in the
Revuz correspondence, if they satisfy, for all $\gamma-$excessive
function $h$, and $f\in\mathcal B_+(\ombar)$,  the
relation:$$\lim_{t\searrow
0}\frac{1}{t}E_{h.m}\left[\int_0^tf(X_s)dA_s\right]=\int_{\po}h(x)(f.\mu)(dx)$$
The family of all equivalence classes of $\ma_c^+(\po)$ and the
family $S(\po)$ are in one to one correspondence under the Revuz
correspondence. In this case, $\mu\in S(\po)$ is called the Revuz
measure of $A$.

\begin{example}
We suppose $\Omega$ to be bounded with Lipschitz boundary.
We have \cite{P}:
$$\lim_{t\searrow
0}\frac{1}{t}E_{h.m}\left[\int_0^tf(X_s)dL_s\right]=\frac{1}{2}\int_{\partial\Omega}h(x)f(x)\sigma(dx)$$
where $L_t$ is the boundary local time of the reflecting Brownian
motion on $\ombar$, and $\sigma$ the surface measure. It follows that $\frac{1}{2}\sigma$ is the
Revuz measure of $L_t$ .
\end{example}

\subsection{General reflected Brownian motion}

Now we turn our attention to the process associated with the
regular Dirichlet form $(\me,\mf)$ on $L^2(\Omega)$ defined by:

\begin{equation}\me(u,v)=\int_{\Omega}\nabla u\nabla v dx\quad,
\mf=\widetilde{H^1}(\Omega)\end{equation}

Due to the Theorem of Fukushima \cite{FOT}, there is a Hunt process
$(X_t)_{t\geq 0}$ associated with it. In addition, $(\me,\mf)$ is
local, thus the Hunt process is in fact a diffusion process (i.e.
A strong Markov process with continuous sample paths). The
diffusion process $M=(X_t,P_x)$ on $\ombar$ is associated with the
the form $\me$ in the sense that the transition semigroup
$p_tf(x)=E_x[f(X_t)]$, $x\in\ombar$ is a version of the
$L^2-$semigroup $\mathcal P_tf$ generated by $\me$ for any
nonnegative $L^2-$function $f$.
We call the diffusion process on $\ombar$ associated with
$(\me,\mf)$ the \emph{General reflecting Brownian motion}.

The process $X_t$ is so named to recall the standard reflecting
Brownian motion in the case of bounded smooth $\Omega$, as the
process associated with $(\me,H^1(\Omega))$. Indeed, when $\Omega$
is bounded with Lipschitz boundary we have that
$\Huntild=H^1(\Omega)$, and by \cite{BH2} the reflecting Brownian
motion $X_t$ admits the following Skorohod representation:

\begin{equation}\label{eq:sko}
X_t=x+W_t+\frac{1}{2}\int_0^t\nu(X_s)dL_s,
\end{equation}
where $W$ is a standard $d-$dimensional Brownian motion, $L$ is the
boundary local(continuous additive functional) associated with
surface measure $\sigma$ on $\po$, and $\nu$ is the inward unit
normal vector field on the boundary.

Now,we apply a general decomposition theorem of additive
functionals to our process $M$. In the same way as in \cite{BH2}, the continuous additive functional
$\widetilde{u}(X_t)-\widetilde{u}(X_0)$ can be decomposed as
follows:
$$\widetilde{u}(X_t)-\widetilde{u}(X_0)=M_t^{[u]}+N_t^{[u]}$$where
$M_t^{[u]}$ is a martingale additive functional of finite energy
and $N_t^{[u]}$ is a continuous additive functional of zero
energy.

Since $(X_t)_{t\geq 0}$ has continuous sample paths, $M_t^{[u]}$
is a continuous martingale whose quadratic variation process is:

\begin{equation}\label{eq:lev}
<M^{[u]},M^{[u]}>_t=\int_0^t|\nabla
u|^2(X_s)ds
\end{equation}

Instead of $u$ we take coordinate function $\phi_i(x)=x_i$. We
have $$X_t=X_0+M_t+N_t$$

We claim that $M_t$ is a Brownian motion with respect to the
filtration of $X_t$. To see that, we use L\'evys criterion. This
follows immediately from \eqref{eq:lev}, which became in the case of
coordinate function: $$<M^{[\phi_i]},M^{[\phi_i]}>=\delta_{ij}t$$

\begin{proposition}
The additive functional $N_t$ is supported by $\po$.
\end{proposition}
\begin{proof}
 See \cite{A}.
\end{proof}

\subsection{Positive smooth measure case}

This subsection is concerned with the probabilistic representation to
the semigroup generated by the Laplacian with general Robin
boundary conditions, which is, actually, obtained by perturbing
the Neumann boundary conditions by a measure. We start with the
Regular Dirichlet form defined by \eqref{eq:fe}, which we call always as
the Dirichlet form associated with the Laplacian with Neumann
boundary conditions.

Let $\mu$ be a positive Radon measure on
$\partial\Omega$ charging no set of zeo relative capacity.
Consider the perturbed Dirichlet form $(\mathcal E_{\mu},\mathcal
F^{\mu})$ on $L^2(\Omega)$ defined by:
$$\mathcal F_{\mu}=\mathcal F \cap L^2(\partial\Omega,\mu)$$
$$\mathcal E_{\mu}(u,v)=\mathcal E(u,v)+\int_{\partial\Omega}uv
d\mu\quad u,v\in\mathcal F^{\mu}$$

We shall see in the following theorem that the transition
function:$$\mathcal P_t^{\mu}f(x)=E_x[f(X_t)e^{-A_t^{\mu}}]$$ is
associated with $(\mathcal E_{\mu},\mathcal F_{\mu})$, where
$A_t^{\mu}$ is a positive additive functional whose Revuz measure
is $\mu$, note that the support of the AF is the same as the
relative quasi-support of its Revuz measure.

\begin{theorem}Let $\mu$ be a positive Radon measure on $\partial\Omega$
charging no set of zero relative capacity and $(A_t^{\mu})_{t\geq
0}$ be its associated PCAF of $(X_t)_{t\geq 0}$. Then $\mathcal
P_t^{\mu}$ is the strongly continuous semigroup associated with
the Dirichlet form $(\mathcal E_{\mu},\mathcal F^{\mu})$ on
$L^2(\Omega)$.
\end{theorem}

The proof of the Theorem 4.2 is similar to the Theorem 6.1.1
\cite{FOT} which was formulated in the first time by S. Albeverio
and Z. M. Ma \cite{AM1} for general smooth measures in the context
of general $(X,m)$. In the case of $X=\overline{\Omega}$, and
working just with measures on $S_0(\partial\Omega)$ the proof
still the same, and works also for any smooth measure concentrated
on $\partial\Omega$. Consequently, the theorem still verified for
smooth measures ''nowhere Radon'' i.e. measures locally infinite
on $\partial\Omega$.

\begin{example} We give some particular examples of $\pmu$:

(1) If $\mu=0$, then $$\mathcal P_t^0 f(x)=E_{x}[f(X_t)] $$ the
semigroup generated by Laplacian with Neumann boundary conditions.

(2) If $\mu$ is locally infinite (nowhere Radon) on $\po$, then
$$\mathcal P_t^{\infty}f(x)=E_{x}[f(B_t)1_{\{t<\tau\}}]$$the semigroup
generated by the Laplacian with Dirichlet boundary conditions (see
Proposition 3.2.1 \cite{Wa}).

(3)  Let $\Omega$ be a bounded and enough smooth to insure the
existence of the surface measure $\sigma$, and $\mu=\beta.\sigma$,
with $\beta$ a measurable bounded function on $\po$, then
$A_t^{\mu}=\int_0^t\beta(X_s)dL_s$, where $L_t$ is the boundary
local time. Consequently :$$\pmu
f(x)=E_x[f(X_t)exp(-\int_0^t\beta(X_s)dL_s)]$$ is the semigroup
generated by the Laplacian with (classical) Robin boundary
conditions gien by $\frac{\partial}{\partial\nu}+\beta=0$.
\end{example}

\begin{proposition}
$\mathcal P_t^{\mu}$  is sub-markovian i.e. $ \pmu\geq 0$
for all $t\geq 0$, and $$||\pmu
f||_{\infty}\leq||f||_{\infty}\quad (t\geq 0)$$
\end{proposition}

\begin{theorem}
Let $\mu\in S(\po)$, then the semigroup $\mathcal P_t^{\mu}$ is
sandwiched between the semigroup of Neumann Laplacian, and the
semigroup of Dirichlet Laplacian. That is :$$0\leq
e^{-t\Delta_D}\leq\mathcal P_t^{\mu}\leq e^{-t\Delta_N}$$for all
$t\geq 0$, in the sense of positive operators.
\end{theorem}

\begin{proposition}Let $\mu,\nu\in S(\partial\Omega)$ such that $\nu\leq\mu$
(i.e $\nu(A)\leq\mu(A),\forall A\in\mathcal B(\partial\Omega)$),
then
$$0\leq e^{-t\Delta_D}\leq\mathcal P_t^{\mu}\leq\mathcal
P_t^{\nu}\leq e^{-t\Delta_N}$$for all $t\geq 0$, in the sense of
positive operators.
\end{proposition}

\section{Kato class of measures}
In this section, we will define a particular Kato class of measures. It is a particular case of Kato class of measures as defined in \cite{AM2} and generalized in \cite{SV}. In fact, to deal with the Robin boundary value problems involving signed smooth measures, one should only consider measures concentrated in the boundary in the theory of Perturbation of Dirichlet forms by measures. After defining this particular Kato class, we will give briefly some of its properties and its analytic description. The proofs are similar, with minor changes, to those of the above cited papers.

\subsection{Definition and properties}

Let $f\in\mathcal B(\partial\Omega)$ and set
\[
 \|f\|_{rq}=\inf_{\Cap(N)=0}\sup_{x\in{\partial\Omega}\setminus N}|f(x)|
\]

where the index 'rq' reminds us ''relatively quasi-everywhere''.

\begin{definition}\label{def:kato} A smooth measure is said to be in Kato class of measures on $\po$, and we denote 
$\mu\in S_K(\partial\Omega)$ if:
\[
 \lim_{t\searrow
0}\|E_{\cdot}A_t^{\mu}\|_{rq}=0
\]
where $A_t^{\mu}$ is the unique
PCAF associated with $\mu$. (Note that $A^{\mu}$ is also supported by $\po$).
\end{definition}

\begin{remark} 1) $S_K(\po)$ is defined in the spirit of
\cite{P,R} and is a particular case of the Kato classe as defined in \cite{AM2}, and generalized in \cite{SV}.

2)Suppose $\Omega$ bounded with Lipschitz boundary. Let $\beta$ be a Borel function on the boundary, and $\sigma$ the surface measure. Define the measure $\mu=\beta\sigma$, then the Definition \ref{def:kato} is exactly \eqref{eq:kat} .

\end{remark}

Now, we give some properties of measures in $S_K(\po)$.

\begin{proposition} Let $\mu\in S_K(\po)$, then there exist
a nonnegative constant $c$ such that:
\[
\|E_{\cdot}e^{A_t^{\mu}}\|_{rq}\leq c
\]
for small $t$.
\end{proposition}

\begin{proof} We have $\mu\in S_K(\partial\Omega)$, then
$\lim_{t\searrow 0}\|E_{\cdot}A_t^{\mu}\|_{rq}=0$. Consequently,
and for $t$ sufficiently small we get
$\|E_{\cdot}A_t^{\mu}\|_{rq}<\alpha<1$, then by Khams'minskii's
lemma  \cite{S, P}, we get $\|E_{\cdot}e^{L_t^{\mu}}\|_{rq}\leq
c$
\end{proof}

\begin{proposition}\label{pro:beta} Let $\mu\in S_K(\po)$, then there exists nonnegative constants $c$ and $\beta$ such that:
\[ 
\|E_{\cdot}e^{A_t^{\mu}}\|_{rq}\leq c e^{\beta t}
\]
for all $t>0$.
\end{proposition}

\begin{proof}
 By the above Proposition we have for some $T$ sufficiently small

\[ 
\|E_{\cdot}e^{A_t^{\mu}}\|_{rq}\leq c \quad\text{ for all }0\leq t\leq T
\]
Let $t>0$ arbitrary, then there exists $n_0\in\mathbb N$ and $t_0<T$ such that $t=n_0T+t_0$. Thus, by the multiplicative propert of $(e^{A_t^{\mu}})$, we get
\[ 
\|E_{\cdot}e^{A_t^{\mu}}\|_{rq}\leq c e^{\beta t}\text{ with } \beta=T^{-1}\log c
\]

\end{proof}

Now, we introduce a subclass of the Kato class of measures, namly,
\[
 S_{K_0}(\po)=\{\mu\in S_K(\po)\cap S_0(\po):\mu(\po)<\infty \}
\]
where $S_0(\po)$ is the totally of the positive Radon measures of finite energy integral as defined in subsection 2.2.. As for Theorem 2.2.4. in \cite{FOT}, Albeverio and Ma have proved an approximation theorem of smooth measures by measures of $S_{K_0}(\po)$, see \cite{AM1} for details.

\subsection{Analytic definition} 

We give here the analytic decription of $S_K(\po)$, the proofs are similar to those in \cite{AM2}. That permits us to treat the signed measures case without recalling probabilistic considerations. We note also that one can define a more general Kato class of measures as defined in \cite{SV}, the so called extended Kato class of measures.

As in \cite{SV}, we introduce the mapping $\Phi(\mu,\alpha)$ for $\mu\in S(\po)$ and $\alpha>0$ by,
\begin{equation}
 \Phi(\mu,\alpha): C_c(\ombar)_+\rightarrow [0,\infty]
\end{equation}
\begin{equation}
 \Phi(\mu,\alpha)f:=\int_{\po}\left(\int_0^{\infty}e^{-\alpha t}p_tf dt\right)d\mu 
\end{equation}
where $(p_t)_{t\geq 0}$ is the markovian transition function of the general reflected brownian motion: The diffusion process associated with $(\me,\mf)$. Recall that $(p_t)_{t\geq 0}$ is a continous version of the semigroup $(\mathcal P_t)_{t\geq 0}$ assciated with $(\me,\mf)$.

\begin{theorem}
 Let $\mu\in S(\po)$. Then the following assertions are equivalent to each other:
\begin{enumerate}
 \item $\mu\in S_K(\po)$,
 \item $\Phi(\mu,\alpha)$ extends to a bounded linear functional on $L^1(\Omega)$ and 

$\lim\limits_{\alpha\to\infty}\|\Phi(\mu,\alpha)\|_1$=0
\end{enumerate}

\end{theorem}

In view of the KLMN theorem, the more suitable class of measures is when the limit in the above theoem is less strictly than $1$. It is exactly, what is done in \cite{SV}, where an extended Kato class of measures was defined. In our context, and to deal just with measures concentrated in $\po$, the definition can be wrote as follow:
\begin{equation}
\begin{array}{ll}
 \hat S_K(\po)=\{\mu\in S(\po):&\exists\alpha>0\text{ s.t. }\Phi(\mu,\alpha)\text{ extends to a bounded linear }\\
&\text{functional on }L^1(\Omega)\text{ and }c(\mu)<1\}
\end{array}
\end{equation}
where $c(\mu):=\lim\limits_{\alpha\to\infty}c_{\alpha}(\mu)$ and $c_{\alpha}(\mu):=\|\Phi(\mu,\alpha)\|_{\infty}=\|\Phi(\mu,\alpha)\|_{L^1(\Omega)'}$.

In the same manner as Theorem 3.3. in \cite{SV}, one can deduce the following theorem:
\begin{theorem}\label{thm:ebound}
 Let $\mu\in\hat S_K(\po)$. Then $\mu$ is $\me-$bounded. More precisely, 
\begin{equation}
 \int_{\po}|u|^2d\mu\leq c(\mu)\me(u,u)+\|u\|_2^2\quad,\forall u\in\mf
\end{equation}

\end{theorem}

\section{Signed measures case}

Let $\mu=\mu^+-\mu^-$ be a signed Broel measure on $\po$. If $|\mu|$ is a smooth measure, then $\mu$ is called a signed smooth measure, and we shall write $\mu\in S(\po)-S(\po)$. It is evident that $\mu\in S(\po)-S(\po)$ if $\mu^+$ and $\mu^-$ are both in $S(\po)$. For $\mu\in S(\po)-S(\po)$ we shal set $A_t^{\mu}:=A_t^{\mu^+}-A_t^{\mu^-}$. We shal call $\mu$ the Revuz measure of $A_t^{\mu}$.

We define for $\mu\in S(\po)-S(\po)$ 
\[
 \me_{\mu}(u,v)=\int_{\Omega}\nabla u\nabla v dx+\int_{\po}uvd\mu,\text{  }\forall u,v\in \mf^{\mu}
\]
where $\mfm=\Huntild\cap L^2(\po,|\mu|)$

First, we begin with the following result

\begin{theorem}\label{thm:infty}
 Let $\mu$ be a signed Borel measure on $\po$ and assume that $|\mu|$ is locally infnite everywhere on $\po$; i.e.,
\[
 \forall x\in\po\text{ and }r>0\quad |\mu|(B(x,r))=\infty.
\]

Then the form $\mem$, which we denote by $\me_{\infty}$, is closed and is given by 
\[
 \me_{\infty}(u,v)=\int_{\Omega}\nabla u\nabla v dx\quad, D(\me_{\infty})=H^1_0(\Omega)
\]

\end{theorem}
\begin{proof}
 Let $u\in\mfm$ and $\widetilde{u}$ its relatively continous version, it follow from the fact that $\int_{\po}|\widetilde{u}|^2d|\mu|<\infty$ that $\widetilde{u}=0$ r.q.e on $\po$. Thus
\[
 \mf^{\infty}:=\mfm=\{u\in\Huntild:\widetilde{u}=0\text{ r.q.e. on }\po\}
\]
One obtain that for all $u,v\in\mf^{\infty}$,
\[
 \me_{\infty}(u,v):=\mem(u,v)=\int_{\Omega}\nabla u\nabla v dx.
\]

Following a characterization of $H^1_0(\Omega)$ with relative capacity \cite[Theorem 2.3.]{AW1}, one conclude that $\mf^{\infty}=H^1_0(\Omega)$.
\end{proof}

\begin{proposition}
 Let $\mu\in S(\po)-S_K(\po)$. Then
\begin{enumerate}
 \item $(\me^{\mu},\mf^{\mu})$ is lower semibounded closed quadratic form,
  \item $(\pmu)_{t\geq 0}$ is the unique strongly continous semigroup corresponding 

to $(\me^{\mu},\mf^{\mu})$,
   \item $\mfm=\mf\cap L^2(\po,\mu^+)$.
\end{enumerate}
\end{proposition}

\begin{proof}
 The proof is the same as Proposition 3.1. in \cite{AM2} with minor change.
\end{proof}

\begin{theorem}
  Let $\mu\in S(\po)-S(\po)$. Then the following assertions are equivalent to each other:
\begin{enumerate}
 \item $(\me^{\mu},\mf^{\mu})$ is lower semibounded,
 \item $(\pmu)_{t\geq 0}$ is a strongly continous semigroup on $L^2(\Omega)$,
  \item there exist constants $c$ and $\beta$ such that 
\[
 \|\pmu f\|_{2}\leq ce^{\beta t}\|f\|_2\quad,\forall f\in L^2(\Omega)
\]
  \item the form $Q_{\mu^-}$ is relatively form bound with respect to  $(\me^{\mu^+},\mf^{\mu^+})$ with bound less than $1$.
\end{enumerate}

If one of the above assertion holds, the closed form coresponding to  $(\pmu)_{t\geq 0}$ is the largest closable form that is smaller that $(\me^{\mu},\mf^{\mu})$.
\end{theorem}

\begin{proof}
 The proof is tha same as Theorem 4.1. in \cite{AM2} with minor change.
\end{proof}

Let $\mu\in S(\po)-S_K(\po)$. We will denote by $\Delta_{\mu}$ the selfadjoint operator in $L^2(\Omega)$ associated with $(\me^{\mu},\mf^{\mu})$; i.e.,

\[
  \left\{
          \begin{array}{ll}
            D(\dmu):=\{u\in \mfm:\exists v\in L^2(\Omega): \mem(u,\varphi)=(v,\varphi)\forall\varphi\in\mfm\}\\
            \dmu u:=-v\\
          \end{array}
  \right.
\]

Since for each $u\in D(\dmu)$ we have 
\begin{equation}\label{eq:lapl}
 \int_{\Omega}\nabla u\nabla\varphi dx+\int_{\po}u\varphi d\mu=\int_{\Omega}v\varphi dx
\end{equation}
for all $\varphi\in\mem$, if we choose $\varphi\in\mathcal D(\Omega)$, the equality \eqref{eq:lapl} can be written
\[
 \langle -\Delta u,\varphi\rangle=\langle v,\varphi\rangle
\]
where $\langle,\rangle$ denotes the duality between $\mathcal D(\Omega)'$ and $\mathcal D(\Omega)$. Since $\varphi\in\mathcal D(\Omega)$ is arbitrary, it follows that 
\[
 -\Delta u=v\quad\text{ in }\mathcal D(\Omega)'
\]
 Thus $\dmu$ is a realization of the Laplacian on $L^2(\Omega)$.

\begin{proposition}
 Let  $\mu\in S(\po)-S_K(\po)$. Then the following assertion are equivalent to each other:
\begin{enumerate}
 \item $(\me^{\mu},\mf^{\mu})$ is regular on $\ombar$.
  \item $|\mu|$ is a Radon measure.
\end{enumerate}
\end{proposition}

\begin{proposition}
  Let  $\mu\in S(\po)-S_K(\po)$. Then there exists a relatively open set $X_0$ satisfying $\Omega\subset X_0\subset\ombar$ such that the form $(\me^{\mu},\mf^{\mu})$ is regular on $X_0$
\end{proposition}

For the proof of the above two propositions, one can follow the proof of Theorem 5.8. in \cite{AM1}. The relatively open set can be explicitly written as follow
\[
 X_0=\ombar\setminus\{x\in\po: |\mu|\left(B(x,r)\right)=\infty, \forall r>0\}
\]

Now define the following subset of $\po$,
\[
 \Gamma^{\infty}:=\{x\in\po: |\mu|\left(B(x,r)\right)=\infty, \forall r>0\}
\]

Note that $\Gamma^{\infty}$ is a relatively closed subset of $\po$. Since $\Gamma^{|\mu|}:=\po\setminus\Gamma^{\infty}$ is a locally compact metric space, it follows from \cite[Theorem 2.18. p.48]{Ru} that $|\mu|_{|\Gamma^{|\mu|}}$ is  a regular Borel measure. Therefore $|\mu|$ is a Radon measure on $\Gamma^{|\mu|}$. As for Theorem \ref{thm:infty}, we have $\widetilde{u}_{|\Gamma_{\infty}}=0$ r.q.e. for each function $u\in\mathcal F^{\mu^+}_{\Gamma^{\infty}}$, where 
\[
 \mathcal F^{\mu^+}_{\Gamma^{\infty}}:=\{u\in\mfmp:\widetilde{u}=0 \text{ r.q.e. on }\Gamma^{\infty}\}
\]

It follows that $\Delta_{\mu}$ is the Laplacian with Dirichlet boundary conditions on $\Gamma^{\infty}$, and with Robin boundary conditions on $\Gamma^{|\mu|}$.
\section{Domination results}

In this section, we will prove a domination theorem. It says that the semigroup $(\emu)_{t\geq 0}$ is sandwitched between $(\emup)_{t\geq 0}$ and $(\emun)_{t\geq 0}$. A very natural quetion arise: Is the converse also true? The answer is yes under a locality assymption.

\begin{definition}
 Let $E$ be an ordered vector space
\begin{enumerate}
 \item $E$ is a vector lattice if any two elements $f,g\in E$ have a supremum, which is denoted by $f\vee g$, and an infinimum, denoted by $f\wedge g$.
  \item Let $E$ be a Banach lattice. A linear subspace $I$ of $E$ is called an ideal if $f\in I$ and $g\in E$ such that $|g|\leq|f|$ imply $g\in I$.
\end{enumerate}

\end{definition}

\begin{theorem}\label{thm:dom}
 Let $\mu\in S(\po)-S_K(\po)$, and $\dmu$ the closed operator on the $L^2(\Omega)$ associated with the closed form $(\me^{\mu},\mf^{\mu})$. Then
\[
 0\leq\emup\leq\emu\leq\emun
\]
for all $t\geq 0$ in the sens of positive operators.
\end{theorem}

 One can see easily from the probabilistic representation of $(\emu)_{t\geq 0}$ that the Propostion is true, but here we will prove it using the following result characterizing domination of positive semigroups due to Ouhabaz and contained in \cite[Th\'eor\`eme 3.1.7.]{Ou},

\begin{theorem}\label{thm:ou}(\textbf{Ouhabaz})
  Let $T$ and $S$ be two positive symmetric $C_0-$semigroups on $L^2(\Omega)$. Let $(a,D(a))$ be the closed form associated with $T$ and $(b,D(b))$ the closed form associated with $S$. Then the following assertions are equivalent.
\begin{enumerate}
 \item $T(t)\leq S(t)$ for all $t\geq 0$ in the sense of positive operators.
  \item $D(a)$ is an ideal of $D(b)$ and $b(u,v)\leq a(u,v)$ for all $u,v\in D(a)_+$.
\end{enumerate}
 
\end{theorem}

\begin{proof}(of Theorem \ref{thm:dom}) Recall that the forms associated with $\emup$, $\emun$ and $\emu$ are given repectively by 
\[
 \me_{\mu^+}(u,v)=\int_{\Omega}\nabla u\nabla v dx+\int_{\po}uvd\mu^+,\text{  }\forall u,v\in \mf^{\mu^+}=\Huntild\cap L^2(\po,\mu^+)
\]
\[
 \me_{\mu}(u,v)=\int_{\Omega}\nabla u\nabla v dx+\int_{\po}uvd\mu,\text{  }\forall u,v\in \mf^{\mu}=\mf^{\mu^+}
\]
and
\[
 \me_{-\mu^-}(u,v)=\int_{\Omega}\nabla u\nabla v dx-\int_{\po}uvd\mu^-,\text{  }\forall u,v\in \mf^{\mu^-}=\Huntild
\]
It is clear that $\mem(u,v)\leq\memp(u,v)$ for all $u,v\in\mfm_+=\mfmp_+$. Then by Theorem \ref{thm:ou} we have $\emup\leq\emu$ for all $t\geq0$ in the sense of positive operator. On the other hand, one have $\memn(u,v)\leq\mem(u,v)$ for all $u,v\in\mfmp_+$. It still to prove that $\mfmp$ is an ideal of $\Huntild$. Let $u\in\mfmp$ and $v\in\Huntild$ such that $0\leq|v|\leq|u|$.We may assume that $u$ and $v$ are r.q.c., it follows that $0\leq|v|\leq|u|$ r.q.e. and therefore $\mu^+-$a.e.(since $\mu$ charges no set of zero relative capacity) . It follows that 
\[
 \int_{\po}|v|^2d\mu^+\leq  \int_{\po}|u|^2d\mu^+<\infty
\]
 and then $v\in L^2(\po,\mu^+)$, which implies that $v\in\Huntild\cap L^2(\po,\mu^+)$.

\end{proof}
For two positive Borel measures $\mu$ and $\nu$ on $\po$ we write $\nu\leq\mu$ if $\nu(A)\leq\mu(A)$ for all $A\in\mathcal B(\po)$, and for two signed Borel measures $\mu$ and $\nu$ on $\po$ we write $\nu\leq\mu$ if $\nu^+\leq\mu^+$ and $\nu^-\geq\mu^-$.
\begin{proposition}
  Let $\mu,\nu\in S(\po)-S_K(\po)$ such that $\nu\leq\mu$. Let $\dmu$ and $\Delta_{\nu}$ denote the selfadjoint operators on $L^2(\Omega)$ associated respectively with the closed forms $(\mem,\mfm)$ and $(\me_{\nu},\mf^{\nu})$. Then
\[
 0\leq\emup\leq\emu\leq e^{-t\Delta_{\nu}}\leq e^{-\Delta_{-\nu^-}}
\]
for all $t\geq 0$ in the sense of positive operators.
\end{proposition} 

\begin{proof}
 It suffices to show that $\emu\leq e^{-t\Delta_{\nu}}$ for all $t\geq 0$ in the sense of positive operators.We have $\nu^+\leq\mu^+$, then $L^2(\po,\mu^+)\subset L^2(\po,\nu^+)$. Thus $\mfmp$ is continuously embedded into $\mathcal F^{\nu^+}$.

We claim that $\mfmp$ is an ideal of $\mathcal F^{\nu^+}$. In fact, let $u\in\mfmp$ and $v\in\mathcal F^{\nu^+}$ be such that $0\leq |v|\leq|u|$. We have to show that $v\in\mfmp$. We may assume that $u$ and $v$ are r.q.c. It is clear that $v\in\Huntild$. Since $0\leq |v|\leq|u|$ a.e., it follows that $0\leq |v|\leq|u|$ r.q.e and therefore $\mu^+$ and $\nu^+$ a.e. (since $\mu^+,\nu^+\in S(\po)$). It then follow
\[
 \int_{\po}|v|^2d\nu^+\leq  \int_{\po}|u|^2d\mu^+<\infty
\]
and therefore $v\in L^2(\po,\mu^+)$ which proves the claim.

Now, let $u,v\in\mfmp_+$. It follows that $u$ and $v$ are positive r.q.e. on $\ombar$ and thus $\mu$ a.e. on $\po$. We have $\nu^+\leq\mu^+$ and $\nu^-\geq\mu^-$, therefore $\mem(u,v)\leq\mathcal E_{\nu}(u,v)$, which completes the proof.
\end{proof}

\begin{proposition}\label{pro:loc}
 Let $\mu\in S(\po)-S_K(\po)$. Then $(\mem,\mfmp)$ is local.
\end{proposition}

\begin{proof}
 The proof is similar to Proposition 3.4.20. in \cite{Wa}.
\end{proof}

The main result of this paper is the converse of Theorem \ref{thm:dom}. More precisely, if $(T(t))_{t\geq 0}$ is a $C_0-$semigroup on $L^2(\Omega)$ satisfaying
\[
 \emup\leq T(t)\leq\emun
\]
for all $t\geq 0$ in the sense of positive operators, under which conditions $T(t)$ is given by a signed measure $\nu$ on $\po$? We suppose that $\Gamma^{\mu}=\po$, we have then the following theorem:

\begin{theorem}\label{thm:dom2}
 Let $\Omega\subset\mathbb R^d$ be an open set and $T=(T(t))_{t\geq 0}$ be a symmetric $C_0-$semigroup on $L^2(\Omega)$ satisfaying
\[
 \emup\leq T(t)\leq\emun
\]
for all $t\geq 0$ in the sense of positive operators, where $\mu^+\in S(\po)$ and $\mu^-\in S_K(\po)$. Let $(\me,D(\me))$ be the closed form on $L^2(\Omega)$ associated with $T$. Suppose in addition that $(\me,D(\me))$ is regular. Then the following assertions are equivalent to each other:
\begin{enumerate}
 \item  $T(t)=e^{-t\Delta_{\nu-\mu^-}}$ for a unique positive Radon measure $\nu$ charging no set of zero relative capacity on $\po$.
 \item $(\me,D(\me))$ is local.
\end{enumerate}
\end{theorem}

\begin{proof}
 (1)$\Rightarrow$(2) This part follows from Proposition \ref{pro:loc}.

(2)$\Rightarrow$(1) We have $D(\me)$ is an ideal of $\Huntild$ and for all $u,v\in D(\me)_+$ we have,
\[
 \int_{\Omega}\nabla u\nabla v dx-\int_{\po}uvd\mu^-\leq\me(u,v)
\]
For $u,v\in D(\me)\cap C_c(\ombar)$ we let
\[
 b(u,v)=\me(u,v)+\int_{\po}uvd\mu^--\int_{\Omega}\nabla u\nabla v dx
\]
Let $\{G_{\beta}^{\me}:\beta>0\}$ be the resolvent of the operator associated with the closed form $(\me, D(\me))$ and $\{G_{\beta}^{-\mu^-}:\beta>0\}$ be the resolvent of $\Delta_{-\mu^-}$. Let $\me^{(\beta)}$ and $\me_{-\mu^-}^{(\beta)}$ be the approximation forms of $\me$ and $\memn$ and let
\begin{equation}
\begin{split}
 b^{(\beta)}(u,v)&:= \me^{(\beta)}(u,v)-\me_{-\mu^-}^{(\beta)}(u,v)\\
                 &=\beta\left(u-\beta G_{\beta}^{\me}u,v\right)-\beta\left(u-\beta G_{\beta}^{-\mu^-}u,v\right)\\
                 &=\beta(\beta(G_{\beta}^{-\mu^-}-G_{\beta}^{\me})u,v)
\end{split}
\end{equation}
Since by domination criterion, $b^{(\beta)}(u,v)\geq 0$ for all positive $u,v\in D(\me)\cap C_c(\ombar)$, we have that $\beta(G_{\beta}^{-\mu^-}-G_{\beta}^{\me})$ is a positive symmetric operator on $L^2(\omega)$ and it then follows from \cite{FOT}(Lemma 1.4.1.) that there exists a unique positive Radon measure $\nu_{\beta}$ on $\ombar\times\ombar$ such that for all $u,v\in D(\me)\cap C_c(\ombar)$ we have
\[
  b^{(\beta)}(u,v)=\beta(\beta(G_{\beta}^{-\mu^-}-G_{\beta}^{\me})u,v)=\beta\int_{\ombar}u(x)v(y)d\nu_{\beta}
\]
It is clear that $b^{(\beta)}(u,v)\to b(u,v)$ as $\beta\nearrow^{\infty}$ for all $u,v\in D(\me)\cap C_c(\ombar)$. Since for each $\beta>0$ and $u\in D(\me)\cap C_c(\ombar)$
\[
 b^{(\beta)}(u,v)\leq\me(u,v)
\]
it follows that the sequence $(\beta\nu_{\beta})$ is uniformly bounded on each compact subsets of $\ombar\times\ombar$ and hence a subsequence converges as $\beta_n\to {\infty}$ vaguely on $\ombar\times\ombar$ to a positive Radon measure $\nu$. The form $(\me,D(\me))$ is regular and then $\nu$ is unique and therefore for all $u,v\in D(\me)\cap C_c(\ombar)$
\[
  b(u,v)=\int_{\ombar}u(x)v(y)d\nu
\]
Since $(\me, D(\me))$ and $(\memn,\Huntild)$ are local, it follows that $b(u,v)=0$ for all $u,v\in D(\me)\cap C_c(\ombar)$ with $\mathrm{supp}[u]\cap\mathrm{supp}[v]=\emptyset$. This implies that $\mathrm{supp}[\nu]\subset\{(x,x):x\in\ombar\}$, and therefore
\[
 b(u,v)=\int_{\ombar}u(x)v(x)d\nu
\]

Since $b(u,v)=0$ for all $u,v\in\mathcal D(\Omega)\subset D(\me)$, we have $\mathrm{supp}[\nu]\subset\ombar\setminus\Omega=\po$ and thus 
\[
b(u,v)=\int_{\po}u(x)v(x)d\nu
\]

Consequently, for all $u,v\in D(\me)\cap C_c(\ombar)$ we have
\[
 \me(u,v)=\int_{\Omega}\nabla u\nabla v dx+\int_{\po}uvd\nu-\int_{\po}uvd\mu^-
\]

The positive Radon measure $\nu$ charges no set of zero relative capacity. In fact, we have $\me(u,u)\leq\memp(u,u)$ for all $u\in\mfmp\subset D(\me)$, which implies
\[
 \int_{\po}|u|^2d\nu\leq\int_{\po}|u|^2d|\mu|
\]
With a particular choice of the function $u$, we have for all Borel subsets $\mathcal O\subset\po$

\[
 \nu(\mathcal O)\leq|\mu|(\mathcal O)
\]

If $\mathcal O$ is of zero relative capacity then $\nu(\mathcal O)=0$, thus $\nu$ also charges no set of zero relative capacity.

To finish , it still to prove that $(\mathcal E,D(\me))=(\mathcal E_{\nu-\mu^-},\mathcal F^{\nu})$.


It is clear that $\mathcal F^{\nu}$is a closed subspace of $D(\me)$. We show that $D(\me)$ is a subspace of $\mathcal F^{\nu}$. Let $u\in D(\me)$. For $n\in\mathbb N$ we let $u_n=u\wedge n$. Then $u_n\in\Huntild$ is relatively quasi-continous. Since $0\leq u_n\leq n$ and $\nu(\po)<\infty$, it follows that $u_n\in L^2(\po,\nu)$ and therefore $u_n\in\mathcal F^{\nu}$. It is also clear that $u_n\to u$ $\Huntild$ and thus after taking a subsequence if necessary, we may assume that $u_n\to u$ r.q.e.(see proposition 2.1. \cite{AW2}). since $\nu$ charges no set of zero relative capacity, it follows that $u_n\to u$ $\nu-$a.e. on $\po$. Finally, since $0\leq u_n\leq k$, the Lebesgue Dominated Convergence Theorem implies that $u_n\to u$ in $L^2(\po,\nu)$ and thus $u_n\to u$ in $\mathcal F^{\nu}$ and therefore $u\in\mathcal F^{\nu}$.

\end{proof}

We can drop out the condition that $(\mem,\mfm)$ is regular, but in this case we shoud add with the locality assymption the fact that $D(\me)\cap C_c(\ombar)$ is dense in $D(\me)$. One can then follow the proof of Theorem 4.1 in \cite{AW2} and the technics in Theorem \ref{thm:dom2} to prove the existence of such measure $\nu$. The inconvenient in this case is that $\nu$ is not necessary unique.


\par\bigskip


\end{document}